\documentclass[11pt,letterpaper]{article}

\usepackage[T1]{fontenc}
\usepackage[utf8]{inputenc}
\usepackage{lmodern}
\usepackage{microtype}
\usepackage[margin=1in]{geometry}
\usepackage{amsmath,amssymb,amsthm,mathtools}
\usepackage{booktabs,tabularx,array}
\usepackage{enumitem}
\usepackage{xcolor}
\usepackage{xurl}
\usepackage{listings}
\usepackage[hidelinks]{hyperref}
\hypersetup{
  pdftitle={A constructive two-parameter Ramsey increment},
  pdfauthor={Milos Tatarevic}
}

\allowdisplaybreaks
\setlist{itemsep=0.25em,topsep=0.4em}

\newtheorem{theorem}{Theorem}[section]
\newtheorem{lemma}[theorem]{Lemma}
\newtheorem{corollary}[theorem]{Corollary}
\theoremstyle{remark}

\lstdefinelanguage{Lean}{
  morekeywords={theorem,def,lemma,structure,namespace,end,where,by,have,show,from,exact,match,with,if,then,else,let,in},
  sensitive=true,
  morecomment=[l]{--},
  morestring=[b]"
}
\title{A constructive two-parameter Ramsey increment}
\author{Milos Tatarevic\\
  \small Impossible Research\\
  \small \texttt{milos.tatarevic@gmail.com}}
\date{}

\begin{document}
\maketitle

\begin{abstract}
We prove the inequality \(R(k+1,s+1)\ge R(k,s)+2k+2s\) for classical Ramsey
numbers, valid for all \(5\le k\le s\), and formalize the proof in Lean~4.
As a consequence, we obtain the new lower bound \(R(12,12)\ge1641\).
\end{abstract}


\section{Introduction}\label{sec:introduction}

The classical Ramsey number \(R(k_1,k_2,\ldots,k_r)\) is the smallest integer
\(n\) such that in any \(r\)-coloring of the edges of the complete graph
\(K_n\) there is a monochromatic copy of \(K_{k_c}\) for some
\(1\le c\le r\).  In this paper we consider two colors, red and blue.  A
coloring with no red \(K_k\) and no blue \(K_s\) will be called a
\emph{\((k,s)\)-graph}.  A critical \((k,s)\)-graph has \(R(k,s)-1\)
vertices.

Xu, Shao, and Radziszowski proved the constructive increment
\cite{XuShaoRadziszowski2011}
\begin{equation}\label{eq:xsr}
  R(a,b+1)\ge R(a,b)+2a-2,
  \qquad a\ge5,\quad b\ge2.
\end{equation}
Applying~\eqref{eq:xsr} twice, with the colors interchanged between the two
applications, gives \(R(12,12)\ge1639\) from the known bound
\(R(11,11)\ge1597\); see \cite{Radziszowski2026}.

In 2017, we constructed a \(1639\)-vertex \((12,12)\)-graph and obtained
\(R(12,12)\ge1640\) \cite{Tatarevic2017}.  The construction used a
one-defect intermediate coloring, complementation, a second structured lift,
and local switching.  It suggested the general inequality
\begin{equation}\label{eq:2017-conjecture}
  R(k+1,s+1)\ge R(k,s)+2k+2s-1,
  \qquad 5\le k\le s,
\end{equation}
but a proof was not found at that time.  We prove the following stronger
statement.

\begin{theorem}\label{thm:main}
For integers \(5\le k\le s\),
\begin{equation}\label{eq:main}
  R(k+1,s+1)\ge R(k,s)+2k+2s.
\end{equation}
By symmetry, the same inequality holds after interchanging \(k\) and \(s\)
whenever \(\min\{k,s\}\ge5\).
\end{theorem}

The proof improves two applications of~\eqref{eq:xsr} by two vertices.  The
first is the connector underlying the 2017 construction.  The second follows
from a stronger certificate-orientation lemma and a branch-dependent edit of
the coupled lifts.  The proof is constructive.  We restate and verify the
special Xu--Shao--Radziszowski lift directly; the published inequality
\eqref{eq:xsr} is used only to obtain the initial disjoint certificates.

The construction has also been formalized in Lean~4
\cite{deMouraUllrich2021,Mathlib2020}.  The formalization covers certificate
extraction, the four orientation cases, both lifts, both connectors, all
clique-exclusion arguments, and the final vertex count.  The Lean sources,
executable construction, verification scripts, and graph certificates are
available in the public repository \cite{TatarevicRepo2026}.

\section{Development of the construction}\label{sec:history}

\subsection{The 2017 construction}

The 2017 bound was not obtained by appending one vertex to an existing
\((12,12)\)-graph.  The construction first modified a
Xu--Shao--Radziszowski lift so that only a controlled defect remained, applied
a complementary structured lift, and then removed the remaining defects by
local switching.  This gave one more vertex than two standard applications
of~\eqref{eq:xsr}, but the construction depended on a particular large
coloring and a search-generated switching pattern.

The proof below identifies the uniform data in that construction.  A red
\(K_{k-1}\) and a disjoint blue \(K_{s-1}\) are used as certificates.  After
possibly complementing the coloring and exchanging the certificates, their
incidence matrix has one of two normal forms, Type~A or Type~B.  The first
connector is common to both forms.  A separate old-edge edit in each form
then permits a second connector.

\subsection{AI-assisted development}

The investigation did not begin with a direct request to prove
\eqref{eq:2017-conjecture}.  It began with a request to review older
Ramsey-search code and constructions, reconstruct the 2017 result from first
principles, examine the same geometry in the smaller Paley \((5,5;37)\) and
\((6,6;101)\) colorings to distinguish structural features from those
dependent on local search, and look for patterns that might lead to more
efficient constructions.  The work alternated between exact computation
and mathematical analysis.  Failed attempts to append a single column or to
remove a common defect core led to the coupled-lift construction used here.

GPT-5.6 Sol, used through Codex, generalized the resulting finite switching
pattern into the first complete version of the proof of
\eqref{eq:2017-conjecture} and its Lean~4 formalization.  After that proof
had been implemented, GPT-5.6 Pro was asked whether the same construction
could contain one more vertex; this led to the strengthened four-way
certificate lemma and the second connector used in Theorem~\ref{thm:main}.
Fable~5 and GPT-5.6 Pro were subsequently used to audit the mathematical
argument.  GPT-5.6 Pro also assisted in preparing and editing this
manuscript.

These systems were used as research, formalization, and writing tools and are
not authors.  The author directed the investigation, selected and checked the
final argument, checked the correspondence between the written and formal
statements, and takes responsibility for the contents of the paper.  Agreement
among the models consulted is not treated as evidence of correctness.  Verification consists of the proof below, the Lean kernel check,
the replay checks described in Section~\ref{sec:lean}, and the exact executable
construction.  Further details of the process and all public artifacts are
available in the repository \cite{TatarevicRepo2026}.

\section{Construction outline}\label{sec:overview}

Let \(G_0\) be a critical \((k_0,s_0)\)-graph, where
\(5\le k_0\le s_0\).  Inequality~\eqref{eq:xsr} and its complement provide a
red \(K_{k_0-1}\) and a disjoint blue \(K_{s_0-1}\).  We use either \(G_0\)
or its color complement as an oriented graph \(G\), with parameters \((k,s)\)
and certificates \(C,I\).  The certificate incidence matrix is then put in
Type~A or Type~B normal form.

The construction has four stages.

\begin{center}
\begin{tabularx}{\textwidth}{@{}>{\raggedright\arraybackslash}p{0.23\textwidth}
                                  >{\raggedright\arraybackslash}p{0.16\textwidth}
                                  X@{}}
\toprule
Stage & Added vertices & Purpose \\
\midrule
First special lift & \(2k-2\) &
Produces a \((k,s+1)\)-graph and a distinguished blue \(K_s\). \\
Complementary lift & \(2s\) &
Lifts the distinguished clique and produces a common
\((k+1,s+1)\)-free base. \\
First connector \(z\) & \(1\) &
Uses the same clone-and-switch construction in Type~A and Type~B. \\
Second connector \(w\) & \(1\) &
Uses a branch-specific old-edge edit followed by projection and
common-neighborhood arguments. \\
\bottomrule
\end{tabularx}
\end{center}

Starting with \(R(k_0,s_0)-1\) vertices, these stages add
\((2k-2)+2s+2=2k_0+2s_0\) vertices.  If the selected orientation is dual, we
complement the final coloring.  The result is a \((k_0+1,s_0+1)\)-graph on
\[
  R(k_0,s_0)+2k_0+2s_0-1
\]
vertices, which proves Theorem~\ref{thm:main}.

\section{Proof}\label{sec:proof}

We now give the construction.  We write \(\overline G\) for color
complementation.  For every switched edge we state its source color; the same
source-color data are used in the Lean formalization.

\subsection{Disjoint certificates and the strengthened orientation lemma}

Fix integers \(5\le k_0\le s_0\). Let \(G_0\) be a critical \((k_0,s_0)\)-graph and put \(N=R(k_0,s_0)-1\). The XSR inequality and its complement give

\[
\begin{aligned}
R(k_0,s_0)&\ge R(k_0,s_0-1)+2k_0-2,\\[0pt]
R(k_0,s_0)&\ge R(k_0-1,s_0)+2s_0-2.
\end{aligned}
\qquad\text{(4)}
\]

Consequently,

\[
|G_0|=R(k_0,s_0)-1\ge R(k_0-1,s_0),
\]

so \(G_0\) contains a red \(K_{k_0-1}\). Choose one and call it \(C_0\). After deleting \(C_0\), at least

\[
R(k_0,s_0)-k_0
 \ge R(k_0,s_0-1)+k_0-2
 \ge R(k_0,s_0-1)
\]

vertices remain. The remainder contains no red \(K_{k_0}\), so it contains a blue \(K_{s_0-1}\); call it \(I_0\). Thus \(C_0\cap I_0=\varnothing\).

We need a stronger ordering than the one used for the one-connector construction. It is convenient to allow a simultaneous interchange of the colors and of the two certificates.

\subsubsection{Oriented certificate systems}

An \textbf{orientation} of the certificate pair means either

\[
(\Gamma,r,t,C,I)=(G_0,k_0,s_0,C_0,I_0)
\]

or

\[
(\Gamma,r,t,C,I)=(\overline{G_0},s_0,k_0,I_0,C_0).
\]

In either orientation, \(\Gamma\) is an \((r,t)\)-graph, \(C\) is a red \(K_{r-1}\), \(I\) is a blue \(K_{t-1}\), the two certificates are disjoint, and \(r,t\ge5\).

After ordering

\[
C=\{c_0,\ldots,c_{r-2}\},
\qquad
I=\{i_0,\ldots,i_{t-2}\},
\]

define

\[
\beta(i_q)=\{p:c_pi_q\text{ is blue}\}.
\]

Every \(\beta(i_q)\) is nonempty, since otherwise \(C\cup\{i_q\}\) would be a red \(K_r\). Every row \(c_p\) has a red neighbor in \(I\), since otherwise \(I\cup\{c_p\}\) would be a blue \(K_t\).

We use the following two normal forms.

\textbf{Type A.} The certificates are ordered so that

\[
\begin{gathered}
c_0i_0\text{ and }c_1i_0\text{ are red},\\[0pt]
P:=\beta(i_0)\ne\varnothing,
\qquad
P\subseteq\{2,\ldots,r-2\},
\qquad
2\in P,
\end{gathered}
\qquad\text{(5A)}
\]

and \(c_1\) is not the unique blue neighbor in \(C\) of any vertex of \(I\).

\textbf{Type B.} Here \(r\le t\), and the certificates are ordered so that

\[
\beta(i_p)=\{p\},
\qquad 0\le p\le r-2.
\qquad\text{(5B)}
\]

\subsubsection{Strengthened orientation lemma}

\begin{lemma}[Strengthened orientation lemma]\label{lem:orientation}
One of the two orientations admits either Type A or Type B. In the Type B
alternative the oriented parameters satisfy \(r\le t\).
\end{lemma}

\begin{proof}
First examine the original orientation. Call a column \(i_q\) \textbf{unique} if \(|\beta(i_q)|=1\). Its sole blue row is its owner.

Suppose a unique column exists. If every row owns at least one unique column, choose one witness for every row. The witnesses are distinct, and because \(k_0\le s_0\) they can be placed in positions \(0,\ldots,k_0-2\). This is Type B.

Otherwise choose a unique column as \(i_0\), call its owner \(c_2\), and choose an uncovered row as \(c_1\). The column \(i_0\) is red to every row except its owner. Since \(r\ge5\), there is another red row, which we call \(c_0\). The remaining rows may be ordered arbitrarily. Then \(P=\beta(i_0)=\{2\}\), and the uncovered row \(c_1\) is not the unique blue neighbor of any column. This is Type A.

Now suppose there is no unique column. If some column has at least two red incidences, it also has at least two blue incidences, because its blue set is nonempty and is not a singleton. Choose two red rows as \(c_0,c_1\), a blue row as \(c_2\), and this column as \(i_0\). After relabeling one blue row as \(c_2\), condition (5A) holds. Since there are no unique columns, \(c_1\) is automatically uncovered. Thus Type A again holds.

The only way the original orientation can fail to yield either form is therefore:

\[
\text{every column has at least two blue and at most one red incidence.}
\qquad\text{(6)}
\]

Apply the same procedure to the dual orientation. A dual column is an original row; blue incidences in the dual are red incidences in the original. If the dual procedure finds a unique or mixed column, it gives Type A or Type B exactly as above. A dual Type B requires \(s_0-1\le k_0-1\); together with \(k_0\le s_0\) this forces \(k_0=s_0\), so its oriented parameters are ordered.

If the dual procedure also failed, then every original row would have at least two red and at most one blue incidence. Counting original blue incidences by columns and by rows would give

\[
2(s_0-1)\le k_0-1,
\]

contrary to \(5\le k_0\le s_0\). Hence at least one orientation has Type A or Type B.
\end{proof}

Choose an orientation and ordering supplied by Lemma~\ref{lem:orientation}. For the remainder of the oriented construction, write its graph and parameters simply as \(G,k,s,C,I\). Thus \(k,s\ge5\); in Type B we additionally have \(k\le s\). At the end we return to the original parameters \(k_0,s_0\).

Define

\[
\begin{aligned}
T&=\{q:c_0i_q\text{ is red}\},\\[0pt]
E&=\{q\in T:c_pi_q\text{ is red for every }p\ge2\},\\[0pt]
S&=T\setminus E,
\qquad
U=\{0,\ldots,s-2\}\setminus T.
\end{aligned}
\qquad\text{(7)}
\]

In Type A, \(E=\varnothing\) and \(0\in S\). Indeed, a vertex counted by \(E\) would have \(c_1\) as its unique blue neighbor in \(C\); and \(i_0\in T\setminus E\) because \(c_0i_0\) is red while \(c_2i_0\) is blue.

In Type B,

\[
0\in U,
\qquad
1\in E,
\qquad
\{2,\ldots,k-2\}\subseteq S.
\qquad\text{(7B)}
\]

For the Type B counts put

\[
E_{\rm lg}=E\cap\{2,\ldots,s-2\},
\qquad
U_{\rm lg}=U\cap\{2,\ldots,s-2\},
\qquad
m=|E_{\rm lg}|.
\]

Among the \(s-3\) ordinary large indices, the \(k-3\) indices \(2,\ldots,k-2\) lie in \(S\). Therefore

\[
m\le s-k,
\qquad
|U_{\rm lg}|\le s-k-m.
\qquad\text{(7a)}
\]

\subsection{The two standard lifts}

Apply the special Xu--Shao--Radziszowski lift to \(G\) and \(C\). In explicit form, add

\[
A=\{a_0,\ldots,a_{k-2}\},\qquad
B=\{b_0,\ldots,b_{k-2}\}.
\]

The colors not already present in \(G\) are:

\begin{enumerate}
\item
  \(A\) is red-complete;
\item
  \(B[\{b_0,b_1\}]\) and \(B[\{b_2,\ldots,b_{k-2}\}]\) are red-complete, while edges between these two parts are blue;
\item
  \(a_pc_p\) is red and every other edge from \(a_p\) to \(G\) is blue;
\item
  \(b_pc_p\) is blue and \(\mathrm{col}(b_pv)=\mathrm{col}(c_pv)\) for \(v\in G-\{c_p\}\);
\item
  \(a_pb_q\) is blue for \(p=q\) and red otherwise.
\end{enumerate}

Call the resulting graph \(L\). For completeness, we verify directly that it is a \((k,s+1)\)-graph.

Let \(K\) be a red clique in \(L\), and put

\[
P=\{p:a_p\in K\},\qquad Q=\{q:b_q\in K\},\qquad D=K\cap G.
\]

The set \(Q\) lies in one of the two \(B\)-parts, and \(P\cap Q\) is empty. If \(|P|\ge2\), then \(D=\varnothing\), because no vertex of \(G\) is red to two distinct vertices of \(A\); hence \(|K|=|P|+|Q|\le k-1\). If \(P=\{p\}\), then \(D\subseteq\{c_p\}\), while \(Q\) avoids \(p\) and has order at most \(k-3\) (the small part has order two and \(k\ge5\)); again \(|K|\le k-1\). Finally, if \(P=\varnothing\), then for every \(q\in Q\), the matched vertex \(c_q\) is not in \(D\), and the clone rule shows that

\[
D\cup\{c_q:q\in Q\}
\]

is a red clique in \(G\). Thus \(|D|+|Q|\le k-1\). Therefore \(L\) has no red \(K_k\).

Now let \(K\) be blue. It contains at most one vertex of \(A\) and at most one vertex from each \(B\)-part. If \(a_p\in K\), then the only possible \(B\)-vertex is \(b_p\). With \(b_p\), the set \(K\cap G\) is disjoint from \(c_p\) and is blue-complete to \(c_p\), so it has order at most \(s-2\); without \(b_p\), it has order at most \(s-1\). In both cases \(|K|\le s\). If \(K\cap A=\varnothing\) and it uses at most one \(B\)-vertex, the same bound is immediate. If it uses two, they have indices in opposite \(B\)-parts. Then \(K\cap G\) excludes each matched \(c_q\), and adjoining either one of them gives a blue clique in \(G\); hence \(|K\cap G|\le s-2\) and again \(|K|\le s\). Thus \(L\) has no blue \(K_{s+1}\).

The set

\[
J=I\cup\{a_0\}
\]

is blue-complete in \(L\). Label its elements by

\[
j_q=i_q\quad(0\leq q\leq s-2),\qquad
j_*=a_0,\quad *=s-1.
\]

Apply the complementary special lift to \(L\) and \(J\). It adds

\[
X=\{x_0,\ldots,x_*\},\qquad
Y=\{y_0,\ldots,y_*\}.
\]

In the original red/blue convention its rules are:

\begin{enumerate}
\item
  \(X\) is blue-complete;
\item
  \(Y[\{y_0,y_1\}]\) and \(Y[\{y_2,\ldots,y_*\}]\) are blue-complete, while edges between these parts are red;
\item
  \(x_rj_r\) is blue and every other edge from \(x_r\) to \(L\) is red;
\item
  \(y_rj_r\) is red and \(\mathrm{col}(y_rv)=\mathrm{col}(j_rv)\) for \(v\in L-\{j_r\}\);
\item
  \(x_ry_t\) is red for \(r=t\) and blue otherwise.
\end{enumerate}

Indeed, \(\overline L\) is an \((s+1,k)\)-graph and \(J\) is a red \(K_s\) in \(\overline L\). Applying the verified lift with \(a=s+1\) and then complementing back proves that the resulting graph \(F\) is a \((k+1,s+1)\)-graph.

\subsection{Add and switch the missing vertex}\label{subsec:add-z}

Add a vertex \(z\). Initially make it a core-clone of \(b_0\):

\[
\mathrm{col}(zv)=\mathrm{col}(b_0v)
\quad(v\in F-\{b_0\}),
\qquad b_0z\text{ red}.
\qquad\text{(8)}
\]

Starting with (8), make the following switches:

\[
\begin{array}{ll}
\text{(i)}&
 a_0y_q:\ {\rm blue}\to{\rm red},\quad
 b_0y_q:\ {\rm red}\to{\rm blue}
 \quad(q\in T),\\[2mm]
\text{(ii)}&
 b_0y_*:\ {\rm blue}\to{\rm red},\qquad
 b_0b_2:\ {\rm blue}\to{\rm red},\\[2mm]
\text{(iii)}&
 b_0x_1,b_0x_*:\ {\rm red}\to{\rm blue},\\[2mm]
\text{(iv)}&
 zx_r:\ {\rm red}\to{\rm blue}
 \quad(r\notin\{1,*\}).
\end{array}
\qquad\text{(9)}
\]

If \(E\ne\varnothing\) (which can happen only in Type B), make the additional switches

\[
\begin{aligned}
a_0x_q &: {\rm red}\to{\rm blue} &&(q\in E),\\[0pt]
y_1y_* &: {\rm red}\to{\rm blue},\\[0pt]
y_1y_q &: {\rm red}\to{\rm blue} &&(q\in E,\ q\ge2).
\end{aligned}
\qquad\text{(10)}
\]

Denote the final graph by \(H\).

\subsection{Switching lemma}\label{subsec:switching}

We prove that \(H\) has no red \(K_{k+1}\) and no blue \(K_{s+1}\). The proof below is included in detail because this is the new part of the argument.

Every switch in (9)--(10) has the displayed source color. Indeed, \(a_0y_q\) is initially blue, \(b_0y_q\) has the color of \(c_0i_q\), \(b_0y_*\) and \(b_0b_2\) are initially blue, every \(b_0x_r\) and \(zx_r\) is initially red, and the exceptional edges in (10) are red by the matching and two-part rules. Thus the construction is well-defined.

Before the switches, adjoining the red-joined clone \(z\) can create no blue \(K_{s+1}\). It can create a red \(K_{k+1}\) only if that clique contains both \(b_0\) and \(z\): if a monochromatic clique contains \(z\) but not \(b_0\), replacing \(z\) by \(b_0\) gives a clique of the same order in \(F\). After the switches, a new forbidden clique must therefore contain one of the following monochromatic edges:

\[
\begin{array}{c|l}
\text{red}
  & b_0z,\ a_0y_q\ (q\in T),\ b_0y_*,\ b_0b_2\\[3pt]
\text{blue}
  & b_0y_q\ (q\in T),\ b_0x_1,\ b_0x_*\\
  & zx_r\ (r\notin\{1,*\}),\ a_0x_q\ (q\in E)\\
  & y_1y_r\quad
    \bigl(E\ne\varnothing,\ 
    r\in(E\cap\{2,\ldots,s-2\})\cup\{*\}\bigr).
\end{array}
\qquad\text{(11)}
\]

For a colored edge \(uv\), a forbidden clique through \(uv\) would require a clique of order \(k-1\) (red) or \(s-1\) (blue) in its common same-color neighborhood. We bound those common-neighborhood clique numbers one row at a time.

\subsubsection{Red edges}

\textbf{(R1) \(b_0z\).} Its common red neighborhood is contained in

\[
N_R^L(b_0)
=N_R^G(c_0)\cup(A-\{a_0\})\cup\{b_1\}.
\]

This set has no red \(K_{k-1}\), since adjoining \(b_0\) would give a red \(K_k\) in \(L\).

\textbf{(R2) \(a_0y_q\), \(q\in T\).} First suppose \(q\in S\). Put

\[
\begin{aligned}
B_q&=\{b_p:p\ne0,\ c_pi_q\text{ is red}\},\\[0pt]
Y_q&=N_R(a_0)\cap N_R(y_q)\cap Y.
\end{aligned}
\]

Every common red neighbor of \(a_0,y_q\) belongs to

\[
\{c_0\}\cup B_q\cup\{x_q\}\cup Y_q.
\qquad\text{(12a)}
\]

A red clique in \(B_q\) uses either \(b_1\), or vertices from \(\{b_2,\ldots,b_{k-2}\}\), but not both. Since \(q\in S\), at least one \(c_pi_q\), \(p\ge2\), is blue, so the latter choice supplies at most \(k-4\) vertices. Thus

\[
\omega_R(B_q)\le\max(1,k-4).
\]

The set \(Y_q\) is blue-complete. In Type A, all its vertices lie in the one \(Y\)-part opposite to \(y_q\). In Type B, \(q\in S\) implies \(q\ge2\), and the only possible common red \(Y\)-vertex from the small part is \(y_1\); there is no second common red \(Y\)-vertex. Finally, \(x_qy_r\) is blue for every \(r\ne q\), so \(x_q\) is blue to every vertex of \(Y_q\). A red clique therefore uses at most one vertex from \(\{x_q\}\cup Y_q\). Its total order is at most

\[
1+\max(1,k-4)+1\leq k-2.
\]

Now let \(q\in E\). Then \(\beta(i_q)=\{1\}\), so \(B_q=\{b_2,\ldots,b_{k-2}\}\), a red clique of order \(k-3\). The switch \(a_0x_q\) removes \(x_q\) from the common red neighborhood. If \(q=1\), the switches in (10) also remove \(y_*\) and every large exceptional \(y_e\). The remaining common \(Y\)-vertices have indices in \(S\), lie in the large \(Y\)-part, and hence are pairwise blue. If one such \(y_r\) is used, the definition of \(S\) gives a \(p\ge2\) with \(c_pi_r\) blue, so \(b_p\) is unavailable; the order is at most

\[
1+(k-4)+1=k-2.
\]

If no such \(Y\)-vertex is used, the order is at most \(1+(k-3)=k-2\). If \(q\ge2\), the only possible common \(Y\)-vertices before the exceptional switch are \(y_0,y_1\); the first is not red to \(a_0\), and the second was removed by (10). The order is again at most \(1+(k-3)=k-2\).

\textbf{(R3) \(b_0y_*\).} Its common red neighborhood is

\[
(A-\{a_0\})\cup\{b_1,b_2\}.
\]

The vertices \(b_1,b_2\) are blue-adjacent; furthermore \(a_pb_p\) is blue. A red clique in this set has order at most \(k-2\).

\textbf{(R4) \(b_0b_2\).} Its common red neighbors lie in

\[
D\cup(A-\{a_0,a_2\})
\cup(X-\{x_1,x_*\})\cup\{y_*\},
\qquad\text{(12)}
\]

where

\[
D=N_R^G(c_0)\cap N_R^G(c_2).
\]

A red clique in \(D\) has order at most \(k-3\), since \(D\cup\{c_0,c_2\}\) would otherwise contain a red \(K_k\) in \(G\). A red clique uses at most one \(X\)-vertex, and no such vertex is red to \(y_*\). Write \(A'=A-\{a_0,a_2\}\) and \(X'=X-\{x_1,x_*\}\). The exhaustive possibilities are:

\begin{itemize}
\item
  if \(y_*\) is used, no vertex of \(D\) or \(X'\) is used, giving at most \(|A'|+1=k-2\);
\item
  if one \(X'\)-vertex and no \(A'\)-vertex are used, \(D\) contributes at most \(k-3\), again giving \(k-2\);
\item
  if one \(X'\)-vertex and one \(a_p\in A'\) are used, at most the single core vertex \(c_p\) can accompany \(a_p\), giving at most \(3\le k-2\);
\item
  if at least two \(A'\)-vertices are used, no vertex of \(D\) can be used, so the clique has at most \(|A'|+1=k-2\) vertices.
\item
  in the remaining cases there is no \(X'\)-vertex or \(y_*\), and either the clique lies in \(D\), or one \(A'\)-vertex restricts its core contribution to the single vertex \(c_p\); these cases are smaller.
\end{itemize}

Thus every subcase has order at most \(k-2\).

Thus no red edge in (11) can lie in a red \(K_{k+1}\).

\subsubsection{Blue edges}

Before (10), the \(Y\)-graph is the disjoint union of two blue cliques, on indices \(\{0,1\}\) and \(\{2,\ldots,*\}\). The only additional blue \(Y\)-edges are those from \(y_1\) to \(y_*\) and to the large exceptional vertices \(y_e\), \(e\in E\), \(e\ge2\). Also, a blue clique contains at most one vertex of \(\{b_2,\ldots,b_{k-2}\}\), since this set is red-complete.

\textbf{(B1) \(b_0y_q\), \(q\in T\).} Its common blue neighbors lie in:

\begin{itemize}
\item
  \(D=N_B^G(c_0)\cap N_B^G(i_q)\);
\item
  the vertices \(b_p\), \(p\ge3\), for which \(c_pi_q\) is blue (a blue clique can use at most one of them, since they lie in the red-complete large \(B\)-part);
\item
  \(x_1,x_*\), excluding \(x_q\);
\item
  the vertices \(y_r\) in the same \(Y\)-part as \(y_q\), excluding \(y_q\), together, when \(E\ne\varnothing\), with the switched cross-part neighbors between \(y_1\) and the large exceptional vertices.
\end{itemize}

Suppose a common blue clique uses an \(X\)-vertex. The only possibilities are \(x_1\) and \(x_*\), with \(x_q\) omitted. Among vertices of \(L-\{b_0\}\), the only blue neighbor of \(x_*\) is \(a_0\), and \(a_0\) is not common to \(b_0,y_q\). Apart from the endpoint \(b_0\), and apart from \(a_0\) in Type B (which is again not common to \(b_0,y_q\)), the only blue neighbor of \(x_1\) in \(L\) is \(i_1\). We now count the compatible \(Y\)-vertices.

\begin{itemize}
\item
  If \(q\ge2\), the ordinary common \(Y\)-vertices are the other \(s-4\) vertices of the large part. Both \(x_1,x_*\) together with these vertices give order at most \(2+(s-4)=s-2\). With only \(x_1\), the possible old vertex \(i_1\) replaces the missing \(x_*\), so the same bound holds. If \(q\in E\), the additional common vertex \(y_1\) is incompatible with \(x_1\), and with \(x_*\) it can be accompanied only by the other large exceptional vertices; (7a) keeps this subcase below \(s-2\).
\item
  If \(q=1\), only \(x_*\) is available. A blue \(Y\)-clique is either \(\{y_0\}\) or a subset of the \(m\) large exceptional vertices, so its total order is at most \(1+\max(1,m)\le s-2\).
\item
  If \(q=0\), necessarily Type A applies. The only common \(Y\)-vertex is \(y_1\); direct inspection gives order at most two, which is at most \(s-2\).
\end{itemize}

Thus every common blue clique using an \(X\)-vertex has order at most \(s-2\).

Suppose no \(X\)-vertex is used. If no \(b_p\) is used, map each \(y_r\) to \(i_r\); together with the core vertices and \(i_q\), this would be a blue clique in \(G\). Hence there are at most \(s-2\) common vertices. If \(b_p\) is used, map it to \(c_p\). The vertex \(c_p\) is not in \(D\), because \(c_0c_p\) is red. The same map, together with \(i_q\), again gives a blue clique in \(G\), so the bound is \(s-2\). For \(q\in E\) there is no such \(b_p\), since \(\beta(i_q)=\{1\}\), so this is an easier subcase.

For completeness, this projection is injective. The vertex \(i_q\) is not in \(D\), and the endpoint \(y_q\) is absent. If \(i_r\in D\), then \(i_ry_r\) is red by the matched-pair rule, so a blue clique cannot contain both vertices. If \(b_p\) is present, its image \(c_p\) is not in \(D\), as observed above; and \(C\cap I=\varnothing\). Every image edge is blue by the clone rules.

\textbf{(B2) \(b_0x_1\).} If \(E=\varnothing\), the common blue neighborhood is contained in

\[
\{i_1,x_*\}\cup(Y-\{y_1,y_*\}).
\]

The vertices \(i_1,x_*\) are red-adjacent, and a blue clique uses only one \(Y\)-part, giving at most \(s-2\).

If \(E\ne\varnothing\), the additional common vertex \(a_0\) can occur, but then only \(Y\)-indices from \(U\) can occur with it. By (7), the large \(Y\)-part omits the safe indices \(2,\ldots,k-2\), and the small part contains only the index 0 besides the exceptional index 1. Without \(a_0\), the previous \(s-2\) bound is unchanged. With \(a_0\), a clique containing both \(a_0,x_*\) has order at most

\[
2+\max(1,|U_{\rm lg}|)\le s-2
\]

by (7a); omitting \(x_*\) cannot increase the bound.

\textbf{(B3) \(b_0x_*\).} Its common blue neighbors lie in

\[
\{a_0,x_1\}\cup\{y_q:0\le q\le s-2\}.
\]

If \(E=\varnothing\), \(a_0x_1\) is red. If \(E\ne\varnothing\), they are blue-adjacent, but their common \(Y\)-indices lie in \(U\), whose large part omits \(2,\ldots,k-2\) by (7). A clique which uses the switched cross-part edges consists of \(y_1\) and large exceptional vertices. It cannot then use either \(a_0\) or \(x_1\), and has order at most \(1+|E\cap\{2,\ldots,s-2\}|\le1+s-k\). All other cliques use one original \(Y\)-part. The large part together with \(x_1\) has order at most \(s-2\); if \(a_0\) is also used, only large \(U\)-indices remain and their number is at most \(s-k\). The small part gives the same or a smaller bound.

\textbf{(B4) \(zx_r\), \(r\notin\{1,*\}\).} Apart from a possible core vertex \(i_r\), a common blue clique is described by disjoint index sets

\[
\begin{aligned}
P&\subseteq\{0,\ldots,*\}\setminus\{1,*,r\}
&&\text{for its }X\text{-vertices},\\[0pt]
Q&\subseteq(\{*\}\cup U)\setminus\{r\}
&&\text{for its }Y\text{-vertices},
\end{aligned}
\]

The set \(Q\) lies in one original \(Y\)-part. In Type A there are no switched cross-part \(Y\)-edges. In Type B, \(1\notin U\), so \(y_1\) is not common to \(z,x_r\), and every switched cross-part edge is irrelevant. Also \(P\cap Q=\varnothing\), because \(x_ty_t\) is red. If \(|P|+|Q|\geq s-1\), the indices 1 and \(*\), neither of which can belong to \(P\), would both have to belong to \(Q\); this is impossible because they lie in opposite \(Y\)-parts. Thus \(|P|+|Q|\leq s-2\).

The core vertex \(i_r\) can occur only for \(r\in U\), and it is blue to no common \(X\)-vertex. If \(r\ge2\), a large-part \(Q\) has at most \(s-3\) vertices because the index \(r\) is absent, while a small-part \(Q\) has at most two; adjoining \(i_r\) therefore gives at most \(s-2\). If \(r=0\), Type A has \(0\in S\), so \(i_0\) does not occur. In Type B, a large-part \(Q\) has order at most \(1+|U_{\rm lg}|\le1+s-k-m\), and the small part is empty after index 0 is removed. Adjoining \(i_0\) is again below \(s-2\).

There is one further possible common vertex: \(a_0\), when \(r\in E\). Here \(r\ge2\). A clique containing \(a_0\) can use only the \(m-1\) other large exceptional \(X\)-vertices. Its common \(Y\)-indices lie in \(U\). With the large \(U\)-part the total order is at most

\[
1+(m-1)+(s-k-m)=s-k,
\]

and with the small part it is at most

\[
1+(m-1)+1=m+1\le s-k+1.
\]

Both bounds are at most \(s-2\).

\textbf{(B5) \(a_0x_q\), \(q\in E\).} A possible common core vertex is only \(i_q\). It is red to every common \(X\)-vertex and to the possible common clone (\(b_0\) for \(q=1\), \(z\) for \(q\ge2\)). Thus a clique containing \(i_q\) consists otherwise of \(Y\)-vertices with indices in \(U\), all in one part. Its order is at most

\[
1+\max\{1,|U_{\rm lg}|\}\le s-2.
\]

Without \(i_q\), the common \(X\)-indices are contained in \(\{*\}\cup(E-\{q\})\), and the common \(Y\)-indices lie in \(U\). In Type B the safe indices \(2,\ldots,k-2\) are absent from both sets. If \(m=|E\cap\{2,\ldots,s-2\}|\), there are at most \(m+1\) common \(X\)-indices, while the large \(U\)-part has at most \(s-k-m\) vertices and the small \(U\)-part consists only of index 0. Thus, without a clone, the large-part total is at most \(s-k+1\) and the small-part total is at most \(m+2\le s-k+2\).

If \(q=1\), the possible clone \(b_0\) is compatible only with \(x_*\) among the common \(X\)-vertices. A clique containing \(b_0\) therefore has order at most

\[
2+\max\{1,|U_{\rm lg}|\}\le s-2,
\]

where the two displayed vertices are \(b_0,x_*\). This estimate is sharp in the smallest case: when \(k=s=5\), the three vertices \(b_0,x_*,y_0\) can all be common blue neighbors.

If \(q\ge2\), the possible clone \(z\) is compatible with the \(m-1\) other large exceptional \(X\)-vertices, but not with \(x_*\) or \(x_1\). With the large \(U\)-part the total is at most

\[
1+(m-1)+(s-k-m)=s-k,
\]

and with the small part it is at most

\[
1+(m-1)+1=m+1\le s-k+1.
\]

Both are at most \(s-2\).

\textbf{(B6) \(y_1y_r\), where \(E\ne\varnothing\) and \(r\in(E\cap\{2,\ldots,s-2\})\cup\{*\}\).} A forbidden clique containing \(b_0\) is impossible by (B1) (and \(b_0y_*\) is red), so assume \(b_0\) is absent. If no \(X\)-vertex is used, keep every vertex of \(L\) fixed and map every \(y_t\) to \(j_t\). No \(j_t\) can already occur together with \(y_t\), because \(j_ty_t\) is red. Thus this maps the two endpoints and their common blue clique injectively to a blue clique in \(L\); order \(s+1\) is impossible.

If \(X\)-vertices are used, let \(P\) be their index set; thus \(P\) avoids \(1,r\). Put

\[
W=\big((E\cap\{2,\ldots,s-2\})\cup\{*\}\big)-\{r\}.
\]

The common \(Y\)-indices are contained in \(W-P\). Hence the total number of \(X\)- and \(Y\)-vertices is at most \(|P|+|W-P|=|P\cup W|\le s-2\), because neither 1 nor \(r\) belongs to \(P\cup W\). No vertex outside \(X\cup Y\) that is blue to both endpoints can be blue to two members of \(P\). Indeed, every vertex of \(L-\{a_0,b_0\}\) is blue to at most one \(X\)-vertex, namely \(x_t\) when the vertex is \(j_t\). The remaining possibilities are unavailable: \(a_0y_1\) is red, \(b_0\) has already been excluded, and \(zy_1\) is red. Thus the case \(|P|\ge2\) has no additional vertex. With one \(X_t\), at most the single old vertex \(j_t\) can also occur. Moreover \(|W|\le s-k\) by (7a). Hence this last subcase has order at most \(1+1+(s-k)\le s-2\).

No blue edge in (11) can therefore lie in a blue \(K_{s+1}\). This proves the switching lemma: \(H\) is a \((k+1,s+1)\)-graph.

Observe that the Type A part of this switching proof used only \(k,s\ge5\). The ordering \(k\le s\) entered only in the Type B bounds (7a). This distinction is needed when Lemma~\ref{lem:orientation} selects the dual Type A orientation.

\subsection{The second connector in Type A}

Assume the oriented certificates have Type A. Retain

\[
P=\beta(i_0),
\qquad
Q=\{2,\ldots,k-2\}\setminus P.
\qquad\text{(13A)}
\]

Thus \(P\ne\varnothing\), \(2\in P\), and \(|Q|\le k-4\).

We will repeatedly use the following elementary collision facts. In a red clique, replacing a present vertex \(b_q\) by \(c_q\) cannot collide with a retained vertex, because \(b_qc_q\) is blue. In a blue clique, replacing a present \(y_q\) by \(i_q\) cannot collide with a retained vertex, because \(y_qi_q\) is red. Adjacency preservation in each projection is checked from the clone and matching rules in the corresponding subcase. Whenever \(z,x_*\), or a small \(B\)-vertex is replaced below, the text records the additional endpoint and collision checks needed for that map.

\subsubsection{Four additional switches}

Starting from the \((k+1,s+1)\)-graph \(H\) constructed in Subsection~\ref{subsec:add-z} and verified in Subsection~\ref{subsec:switching}, make the following additional switches among its old vertices:

\[
\begin{aligned}
c_px_* &: \text{red}\longrightarrow\text{blue}
&& (p\in P),\\[0pt]
i_0a_2 &: \text{blue}\longrightarrow\text{red},\\[0pt]
y_0y_* &: \text{red}\longrightarrow\text{blue},\\[0pt]
y_0y_2 &: \text{red}\longrightarrow\text{blue}.
\end{aligned}
\qquad\text{(14A)}
\]

Call the resulting graph on the old vertex set \(\widetilde H_A\). Every displayed source color is forced by the lift rules:

\begin{itemize}
\item
  among vertices of \(L\), \(x_*\) is blue only to \(j_*=a_0\), so \(c_px_*\) is red;
\item
  among vertices of \(G\), \(a_2\) is red only to \(c_2\), so \(i_0a_2\) is blue;
\item
  \(y_0\) lies in the small \(Y\)-part and \(y_2,y_*\) lie in the large part, so both cross-part edges are red.
\end{itemize}

\subsubsection{\texorpdfstring{Add the vertex \(w\)}{Add the vertex w}}

Add one vertex \(w\). Its colors to the original core \(G\) are

\[
\begin{aligned}
wc_p&\text{ is red} &&(0\le p\le k-2),\\[0pt]
wi_0&\text{ is blue},\\[0pt]
\mathrm{col}(wv)&=\mathrm{col}(i_0v)
&&\bigl(v\in G\setminus(C\cup\{i_0\})\bigr).
\end{aligned}
\qquad\text{(15A)}
\]

Its colors to the fringe are:

\[
\begin{array}{c|c}
\text{blue from }w&\text{red from }w\\[3pt]
a_2&A\setminus\{a_2\}\\
 b_0,b_1,\ \{b_p:p\in P\}&\{b_q:q\in Q\}\\
 x_0,\ldots,x_{s-2}&x_*\\
 \{y_q:0\le q\le s-2,\ q\notin\{0,2\}\}&y_0,y_2,y_*\\
 &z
\end{array}
\qquad\text{(16A)}
\]

Denote the final graph by \(H_A^+\).

\subsubsection{The additional old-edge switches are safe}\label{subsec:typeA-old-edges}

Since \(H\) is already a \((k+1,s+1)\)-graph, a forbidden clique of \(\widetilde H_A\) must contain an edge in (14A). We inspect the common same-color neighborhoods in \(\widetilde H_A\).

\textbf{(A-R) The red edge \(i_0a_2\).} Its common red neighborhood is contained in

\[
\{b_p:p\notin P\}
 \cup (X\setminus\{x_0\})\cup\{z\}.
\qquad\text{(17A)}
\]

The relevant \(B\)-vertices are the red small pair \(b_0,b_1\) and the red large clique \(\{b_q:q\in Q\}\); no red clique uses both parts. The special switch \(b_0b_2\) is irrelevant because \(2\notin Q\). A red clique uses at most one \(X\)-vertex. The vertex \(z\) is blue to every \(b_q\), \(q\in Q\), and is blue to every \(X\)-vertex except \(x_1,x_*\); for those two exceptions \(b_0\) is blue to the \(X\)-vertex. Therefore a common red clique has order at most

\[
\max\{3,\ |Q|+1\}\le k-2.
\]

Thus \(i_0a_2\) lies in no red \(K_{k+1}\).

\textbf{(A-B1) A blue edge \(c_px_*\), \(p\in P\).} Its common blue neighborhood is contained in

\[
\{a_0\}\cup\{y_q:c_pi_q\text{ is blue}\}.
\qquad\text{(18A)}
\]

Among the displayed \(Y\)-vertices, a blue clique is contained in one original \(Y\)-part, except for the single new cross-part pair \(y_0y_2\). Hence it has order at most \(s-3\). Including \(a_0\) gives at most \(s-2\) common blue vertices, so no such edge lies in a blue \(K_{s+1}\).

\textbf{(A-B2) The blue edge \(y_0y_*\).} Its common blue neighborhood is

\[
N_B^G(i_0)
 \cup (X\setminus\{x_0,x_*\})\cup\{y_2\}.
\qquad\text{(19A)}
\]

Let \(K\) be a blue clique in this set.

\begin{itemize}
\item
  If \(K\) uses at least two \(X\)-vertices, it uses no old vertex of \(G\). Its \(X\)-indices, together with the possible index 2 supplied by \(y_2\), are distinct; hence \(|K|\le s-2\).
\item
  If it uses exactly one \(x_r\), it uses at most the matched old vertex \(i_r\) and possibly \(y_2\). Thus \(|K|\le3\le s-2\).
\item
  If it uses no \(X\)-vertex and omits \(y_2\), then \(K\subseteq N_B^G(i_0)\). Adjoining \(i_0\) gives a blue clique in \(G\), so \(|K|\le s-2\).
\item
  If it uses \(y_2\) and no \(X\)-vertex, its core part lies in \(N_B^G(i_0)\cap N_B^G(i_2)\). Adjoining both \(i_0,i_2\) gives a blue clique in \(G\), so the core part has order at most \(s-3\) and again \(|K|\le s-2\).
\end{itemize}

\textbf{(A-B3) The blue edge \(y_0y_2\).} Put

\[
D=N_B^G(i_0)\cap N_B^G(i_2),
\]

and

\[
B_{02}=\{b_p:p\ge2,\ c_pi_0\text{ and }c_pi_2\text{ are blue}\}.
\]

Its common blue neighborhood is contained in

\[
D\cup(A\setminus\{a_0\})\cup\{b_0\}\cup B_{02}
 \cup(X\setminus\{x_0,x_2\})\cup\{y_*\}.
\qquad\text{(20A)}
\]

A blue clique uses at most one vertex of \(A\) and at most one vertex of \(B_{02}\).

If at least two \(X\)-vertices occur, no ordinary old vertex can occur. The only exception is \(b_0\), and it is blue to two \(X\)-vertices only for the pair \(x_1,x_*\); in that case \(y_*\) is unavailable. Otherwise the \(X\)-indices and the possible \(y_*\) index give at most \(s-2\) vertices. Hence this case is safe.

With exactly one \(X\)-vertex \(x_r\) and \(r\ne *\), at most its matched old vertex \(i_r\) can occur, with the possible additional exception \(b_0\) when \(r=1\). The edge \(b_0y_*\) is red, so together with the possible \(y_*\) the total is at most three. If the sole \(X\)-vertex is \(x_*\), then \(y_*\) is unavailable. The switches \(c_px_*\) allow core vertices \(c_p\) with \(p\in P\), but these form a red clique, and none is blue to \(b_0\); hence at most one additional vertex occurs. In every case the order is at most three, and therefore at most \(s-2\).

Suppose no \(X\)-vertex occurs. If \(y_*\) occurs, no \(A\)- or \(B\)-vertex occurs, and the remaining vertices lie in \(D\). Adjoining \(i_0,i_2\) to the core part gives a blue clique in \(G\); therefore the total is at most \(s-2\).

Finally suppose neither an \(X\)-vertex nor \(y_*\) occurs. If \(b_0\) is absent, keep the vertices of \(D\cup A\cup B_{02}\) in the first-lift graph \(L\) and adjoin \(i_0,i_2\). This gives a blue clique in \(L\) of order \(|K|+2\). Suppose \(b_0\) occurs. The only \(A\)-vertex blue to \(b_0\) is \(a_0\), but \(a_0\) is not in the common neighborhood in (20A); hence \(K\) contains no \(A\)-vertex. If no large \(B\)-vertex occurs, replace \(b_0\) by \(a_0\) before adjoining \(i_0,i_2\). If \(b_0,b_p\) occur with \(b_p\in B_{02}\), replace \(b_0\) by \(a_p\). Moreover \(c_p\notin K\) because \(b_0c_p\) is red, so \(a_p\) is blue to every retained core vertex. In all cases the resulting set is an injective blue clique of order \(|K|+2\) in \(L\). Since \(L\) has no blue \(K_{s+1}\), we have \(|K|\le s-2\).

Thus \(\widetilde H_A\) remains a \((k+1,s+1)\)-graph.

\subsubsection{\texorpdfstring{Red cliques through \(w\) in Type A}{Red cliques through w in Type A}}\label{subsec:typeA-red-w}

The red neighborhood of \(w\) is

\[
\begin{aligned}
N_R(w)=\;&C
 \cup\bigl(N_R^G(i_0)\setminus C\bigr)
 \cup(A\setminus\{a_2\})\\[0pt]
&\cup\{b_q:q\in Q\}
 \cup\{x_*,y_0,y_2,y_*,z\}.
\end{aligned}
\qquad\text{(21A)}
\]

The three vertices \(y_0,y_2,y_*\) are pairwise blue after (14A), so a red clique \(K\subseteq N_R(w)\) uses at most one of them.

\textbf{A red clique containing \(y_0\).} The vertex \(y_0\) is blue to \(x_*\) and to every \(A\)-vertex except \(a_0\). First suppose \(a_0\) is absent. If \(z\) is also absent, map

\[
y_0\longmapsto i_0,
\qquad
b_q\longmapsto c_q\quad(q\in Q),
\]

and keep every core vertex fixed. The map is injective: a red clique cannot contain both \(b_q\) and \(c_q\) because \(b_qc_q\) is blue, and \(i_0\notin K\) because \(wi_0\) is blue. The clone rules show that the image is a red clique in \(G\) of the same order. Hence \(|K|\le k-1\). If \(z\) occurs, no \(b_q\), \(q\in Q\), occurs because \(zb_q\) is blue. Map \(y_0\mapsto i_0\) and \(z\mapsto c_0\), keeping the core fixed. The vertices \(c_0,i_0\) are red-adjacent, and the image is again an injective red clique in \(G\). Thus \(|K|\le k-1\).

It remains to allow \(a_0\). Then \(z\) is unavailable, and \(c_0\) is the only possible core vertex because \(a_0\) is blue to every vertex of \(G-\{c_0\}\). The remaining possible vertices are in the red clique \(\{b_q:q\in Q\}\). Consequently

\[
|K|\le 1+1+1+|Q|\le k-1,
\]

where the first three terms account for \(y_0,a_0,c_0\).

\textbf{A red clique containing \(y_2\).} Suppose first that \(a_0\) is absent. If \(z\) is absent, map \(y_2\mapsto i_2\) and every occurring \(b_q\mapsto c_q\), keeping the core fixed. If \(z\) occurs, then no \(b_q\), \(q\in Q\), occurs because \(zb_q\) is blue; map \(y_2\mapsto i_2\) and \(z\mapsto c_0\). The required edge \(c_0i_2\) is red precisely because \(zy_2\) is red, and \(c_0\) cannot collide with a retained core vertex because \(zc_0\) is blue. Thus both maps are injective red-clique projections into \(G\). If \(a_0\) occurs, then \(z\) is unavailable and only \(c_0\) can occur from the core. At most the \(|Q|\) large \(B\)-vertices can occur in addition, so again

\[
|K|\le 1+1+1+|Q|\le k-1.
\]

\textbf{A red clique containing \(y_*\).} Here \(z\) is unavailable, and the only possible core vertex is \(c_0\). The \(A\)-indices available from the red neighborhood of \(w\) are all indices except 2; the \(B\)-indices are \(Q\); and a red clique cannot contain both \(a_q\) and \(b_q\).

If \(c_0\) occurs, at most \(a_0\) occurs from \(A\); if \(x_*\) also occurs, even \(a_0\) is unavailable. Since \(|Q|\le k-4\), the order is at most \(k-1\). If \(c_0\) is absent and \(x_*\) is absent, the union of available \(A\)- and \(B\)-indices has order at most \(k-2\); adjoining \(y_*\) gives at most \(k-1\). If \(x_*\) occurs, index 0 is additionally unavailable to \(A\), so the index union has order at most \(k-3\); together with \(x_*,y_*\) the order is again at most \(k-1\).

\textbf{A red clique containing no \(Y\)-vertex.} We split according to \(x_*\) and \(z\).

\begin{itemize}
\item
  If neither occurs, the clique lies in the first-lift graph \(L\), and therefore has order at most \(k-1\).
\item
  If \(z\) occurs but \(x_*\) does not, replace \(z\) by \(b_0\). No \(b_q\), \(q\in Q\), occurs with \(z\), and the exceptional switch \(b_0b_2\) is irrelevant because \(2\notin Q\). Thus this is an injective red clique of the same order in \(L\).
\item
  Suppose \(x_*\) occurs but \(z\) does not. If no \(A\)-vertex occurs, replace \(x_*\) by \(i_0\). Every core vertex \(c_p\) with \(p\in P\) was excluded by the switch \(c_px_*\), and all other possible core and \(B_Q\) vertices are red to \(i_0\); hence the image is a red clique in \(L\). If at least two \(A\)-vertices occur, no core vertex occurs and the matching rule, together with the unavailable indices 0 and 2, gives at most \(k-3\) vertices from \(A\cup B_Q\); adding \(x_*\) is safe. If exactly one \(A\)-vertex occurs, at most its matched core vertex can occur, and \(|Q|\le k-4\); the total is at most \(k-1\).
\item
  Finally suppose both \(x_*\) and \(z\) occur. Then no \(B_Q\)-vertex occurs. With no \(A\)-vertex, map \(x_*\mapsto i_0\) and \(z\mapsto c_0\), obtaining a red clique in \(G\). With at least two \(A\)-vertices, no core vertex occurs and at most \(k-3\) such \(A\)-vertices are available; with exactly one \(A\)-vertex, at most one core vertex occurs. The respective bounds are \(k-1\) and \(4\le k-1\).
\end{itemize}

Thus the red neighborhood of \(w\) contains no red \(K_k\), and no red \(K_{k+1}\) of \(H_A^+\) contains \(w\).

\subsubsection{\texorpdfstring{Blue cliques through \(w\) in Type A}{Blue cliques through w in Type A}}

The blue neighborhood of \(w\) is

\[
\begin{aligned}
N_B(w)=\;&\{i_0\}
 \cup\bigl(N_B^G(i_0)\setminus C\bigr)
 \cup\{a_2,b_0,b_1\}\\[0pt]
&\cup\{b_p:p\in P\}
 \cup\{x_0,\ldots,x_{s-2}\}\\[0pt]
&\cup\{y_q:0\le q\le s-2,\ q\notin\{0,2\}\}.
\end{aligned}
\qquad\text{(22A)}
\]

Let \(K\) be a blue clique in this neighborhood.

\textbf{At least two \(X\)-vertices.} No core vertex, \(a_2\), \(b_1\), or large \(B\)-vertex is blue to two distinct ordinary \(X\)-vertices. The only exception is \(b_0\), which is blue to \(x_1,x_*\) in \(H\), but \(x_*\) is not in \(N_B(w)\). Hence no old vertex occurs. If \(P_X\) and \(P_Y\) are the index sets of the \(X\)- and \(Y\)-vertices, then \(P_X\cap P_Y=\varnothing\) because \(x_qy_q\) is red. Both are subsets of the \(s-1\) ordinary indices, so \(|K|\le s-1\).

\textbf{Exactly one \(X\)-vertex, say \(x_r\).} At most the matched core vertex \(i_r\) occurs. For \(r=1\), the switched neighbor \(b_0\) may also occur, but the compatible \(Y\)-indices then exclude \(0,1,2\) and number at most \(s-4\). Thus this subcase has order at most \(1+2+(s-4)=s-1\). For \(r=0\) or \(r=2\), there are at most \(s-3\) compatible \(Y\)-vertices and one matched old vertex; for every other \(r\) there are at most \(s-4\). All cases give \(|K|\le s-1\).

\textbf{No \(X\)-vertex.} Map every occurring \(y_q\) to \(i_q\) and keep all core vertices fixed. This map is injective because \(i_qy_q\) is red. A blue clique uses at most one of \(b_0,b_1\) and at most one large \(B\)-vertex.

First suppose no small \(B\)-vertex occurs. If \(i_0\notin K\), keep a possible \(a_2\) and a possible large \(b_p\), and adjoin \(i_0\) after the \(Y\)-projection. In the original first-lift graph \(L\), every retained vertex is blue to \(i_0\): in particular \(a_2i_0\) is blue in \(L\), and \(b_pi_0\) is blue for \(p\in P\). We obtain a blue clique of order \(|K|+1\) in \(L\), so \(|K|\le s-1\).

If \(i_0\in K\), the switch \(i_0a_2\) excludes \(a_2\). If no large \(B\)-vertex occurs, the projected set is a blue clique in \(G\). If \(b_p\), \(p\in P\), occurs, replace it by \(c_p\). No collision is possible because every \(c_p\) is red to \(w\) and therefore absent from \(K\). The image is an injective blue clique in \(G\). Again \(|K|\le s-1\).

Finally suppose one small vertex \(b_r\), \(r\in\{0,1\}\), occurs. Then \(i_0\) and \(a_2\) are absent because both \(c_0i_0\) and \(c_1i_0\) are red and \(a_2b_r\) is red. If no large \(B\)-vertex occurs, replace \(b_r\) by \(a_r\); if \(b_p\), \(p\in P\), also occurs, replace \(b_r\) by \(a_p\) and retain \(b_p\). In either case adjoin \(i_0\). The core contains no vertex of \(C\) because \(w\) is red to all of \(C\), so the new \(A\)-vertex is blue to every retained core vertex. The result is an injective blue clique of order \(|K|+1\) in \(L\). Therefore \(|K|\le s-1\).

The blue neighborhood of \(w\) contains no blue \(K_s\). Together with Subsections~\ref{subsec:typeA-old-edges} and~\ref{subsec:typeA-red-w}, this proves that \(H_A^+\) is a \((k+1,s+1)\)-graph.

\subsection{The second connector in Type B}

Assume now that the oriented certificates have Type B, so

\[
\beta(i_p)=\{p\}
\qquad(0\le p\le k-2),
\]

and \(5\le k\le s\). Start from the Type B graph \(H\) constructed in Subsection~\ref{subsec:add-z} and verified in Subsection~\ref{subsec:switching}.

\subsubsection{Construction}

Make two additional switches among the old vertices:

\[
c_2x_*:\text{red}\longrightarrow\text{blue},
\qquad
i_2a_2:\text{blue}\longrightarrow\text{red}.
\qquad\text{(14B)}
\]

The source colors follow from the lift rules. Call the resulting old graph \(\widetilde H_B\).

Add a vertex \(w\). On the original core put

\[
\begin{aligned}
wc_p&\text{ is red} &&(0\le p\le k-2),\\[0pt]
wi_2&\text{ is blue},\\[0pt]
\mathrm{col}(wv)&=\mathrm{col}(i_2v)
&&\bigl(v\in G\setminus(C\cup\{i_2\})\bigr).
\end{aligned}
\qquad\text{(15B)}
\]

On the fringe put

\[
\begin{array}{c|c}
\text{blue from }w&\text{red from }w\\[3pt]
a_2&A\setminus\{a_2\}\\
 b_0,b_1,b_2&\{b_p:3\le p\le k-2\}\\
 x_0,\ldots,x_{s-2}&x_*\\
 \{y_q:0\le q\le s-2,\ q\ne2\}&y_2,y_*\\
 &z
\end{array}
\qquad\text{(16B)}
\]

Call the final graph \(H_B^+\).

\subsubsection{The two additional old-edge switches are safe}

\textbf{(B-R) The red edge \(i_2a_2\).} Its common red neighborhood is contained in

\[
(B\setminus\{b_2\})
 \cup(X\setminus\{x_2\})\cup\{z\}.
\qquad\text{(17B)}
\]

The small pair \(b_0,b_1\) and the large clique \(\{b_3,\ldots,b_{k-2}\}\) cannot be mixed in a red clique, and a red clique uses at most one \(X\)-vertex. The vertex \(z\) is blue to the large \(B\)-vertices and to every \(X\)-vertex except \(x_1,x_*\); for those two exceptions \(b_0\) is blue to the \(X\)-vertex. Hence a common red clique has order at most

\[
\max\{3,(k-4)+1\}\le k-2.
\]

\textbf{(B-B) The blue edge \(c_2x_*\).} Its common blue neighborhood is contained in

\[
\{a_0\}\cup\{y_q:c_2i_q\text{ is blue}\}.
\qquad\text{(18B)}
\]

The witness identities \(\beta(i_0)=\{0\}\) and \(\beta(i_1)=\{1\}\) exclude \(y_0,y_1\). Hence every displayed \(Y\)-vertex lies in the original large \(Y\)-part, and a blue clique there has order at most \(s-3\). With \(a_0\) the total is at most \(s-2\). The exceptional switched edges incident with \(y_1\) are irrelevant because \(y_1\) is not common.

Therefore \(\widetilde H_B\) remains a \((k+1,s+1)\)-graph.

\subsubsection{\texorpdfstring{Red cliques through \(w\) in Type B}{Red cliques through w in Type B}}

The red neighborhood is

\[
\begin{aligned}
N_R(w)=\;&C
 \cup\bigl(N_R^G(i_2)\setminus C\bigr)
 \cup(A\setminus\{a_2\})\\[0pt]
&\cup\{b_q:3\le q\le k-2\}
 \cup\{x_*,y_2,y_*,z\}.
\end{aligned}
\qquad\text{(19B)}
\]

The vertices \(y_2,y_*\) are blue-adjacent, so a red clique uses at most one of them.

If \(y_2\) occurs, \(x_*\) is unavailable and the only possible \(A\)-vertex is \(a_0\). First suppose \(a_0\) is absent. Without \(z\), map \(y_2\mapsto i_2\) and every occurring \(b_q\mapsto c_q\), keeping the core fixed. With \(z\), no such \(b_q\) occurs because \(zb_q\) is blue; map \(y_2\mapsto i_2\) and \(z\mapsto c_0\). Here \(c_0i_2\) is red, and \(c_0\) cannot collide with the retained core because \(zc_0\) is blue. In both cases the image is an injective red clique in \(G\) of the same order. If \(a_0\) occurs, then \(z\) is unavailable and only \(c_0\) can occur from the core. There are at most \(k-4\) red \(B\)-vertices, so

\[
|K|\le 1+1+1+(k-4)=k-1.
\]

If \(y_*\) occurs, \(z\) is absent and the only possible core vertex is \(c_0\). The available \(A\)-indices are all indices except 2, while the available \(B\)-indices are \(3,\ldots,k-2\); a red clique cannot use both \(a_q\) and \(b_q\). If \(c_0\) occurs, at most \(a_0\) occurs from \(A\); if \(x_*\) also occurs, \(a_0\) is unavailable. Since there are \(k-4\) available \(B\)-indices, the order is at most \(k-1\). If \(c_0\) is absent and \(x_*\) is absent, at most \(k-2\) vertices occur from \(A\cup B\), and \(y_*\) gives the final vertex. If \(x_*\) occurs, index 0 is also unavailable to \(A\), so at most \(k-3\) vertices occur from \(A\cup B\); together with \(x_*,y_*\) this again gives \(k-1\).

Suppose no \(Y\)-vertex occurs. We split according to \(x_*\) and \(z\).

\begin{itemize}
\item
  If neither occurs, the clique lies in \(L\) and has order at most \(k-1\).
\item
  If \(z\) occurs but \(x_*\) does not, no red \(B\)-vertex can occur with \(z\). Replace \(z\) by \(b_0\). The switch \(b_0b_2\) is irrelevant because \(b_2\) is not red to \(w\), so this is an injective red clique of the same order in \(L\).
\item
  Suppose \(x_*\) occurs but \(z\) does not. If no \(A\)-vertex occurs, replace \(x_*\) by \(i_2\). The switch \(c_2x_*\) excludes \(c_2\), and every other possible core or red \(B\)-vertex is red to \(i_2\); the image is a red clique in \(L\). If at least two \(A\)-vertices occur, no core vertex occurs. The unavailable indices 0 and 2 leave at most \(k-3\) vertices in \(A\cup B\), so adjoining \(x_*\) is safe. With exactly one \(A\)-vertex, at most its matched core vertex occurs; together with at most \(k-4\) red \(B\)-vertices and \(x_*\), the order is at most \(k-1\).
\item
  Finally suppose both \(x_*\) and \(z\) occur. No red \(B\)-vertex occurs. With no \(A\)-vertex, map \(x_*\mapsto i_2\) and \(z\mapsto c_0\), obtaining an injective red clique in \(G\). With at least two \(A\)-vertices, no core vertex occurs and at most \(k-3\) such \(A\)-vertices are available, giving at most \(k-1\) vertices in total. With exactly one \(A\)-vertex, at most one core vertex occurs, and the total is at most \(4\le k-1\).
\end{itemize}

Every no-\(Y\) subcase therefore has order at most \(k-1\).

Hence no red \(K_{k+1}\) containing \(w\) exists.

\subsubsection{\texorpdfstring{Blue cliques through \(w\) in Type B}{Blue cliques through w in Type B}}

The blue neighborhood is

\[
\begin{aligned}
N_B(w)=\;&\{i_2\}
 \cup\bigl(N_B^G(i_2)\setminus C\bigr)
 \cup\{a_2,b_0,b_1,b_2\}\\[0pt]
&\cup\{x_0,\ldots,x_{s-2}\}
 \cup\{y_q:0\le q\le s-2,\ q\ne2\}.
\end{aligned}
\qquad\text{(20B)}
\]

The only blue \(Y\)-edges between the original small and large parts are those from \(y_1\) to the large exceptional vertices. Their number is controlled by

\[
m=|E_{\rm lg}|\le s-k.
\]

Let \(K\) be a blue clique in \(N_B(w)\).

If at least two ordinary \(X\)-vertices occur, no old vertex occurs. The \(X\)- and \(Y\)-index sets are disjoint by the red diagonal edges \(x_qy_q\), so \(|K|\le s-1\).

Suppose exactly one \(x_r\) occurs. At most the matched core vertex \(i_r\) occurs, with the possible additional switched neighbor \(b_0\) when \(r=1\). A blue clique in the available \(Y\)-vertices is one of:

\begin{itemize}
\item
  a subset of the small pair \(\{y_0,y_1\}\);
\item
  a subset of the large ordinary indices \(3,\ldots,s-2\);
\item
  \(y_1\) together with large exceptional vertices, of order at most \(1+m\le s-4\).
\end{itemize}

For \(r=1\), the diagonal removes \(y_1\), leaving at most \(s-4\) compatible \(Y\)-vertices, so even allowing both \(i_1,b_0\) gives \(|K|\le s-1\). For \(r\ne1\), at most one old vertex occurs and the largest available \(Y\)-clique has order at most \(\max\{2,s-4\}\); since \(s\ge5\), the total is again at most \(s-1\).

Finally suppose no \(X\)-vertex occurs. Map every \(y_q\) to \(i_q\); the exceptional cross-part edges map to edges inside the blue clique \(I\), and the map is injective because \(i_qy_q\) is red.

If no small \(B\)-vertex occurs and \(i_2\notin K\), keep a possible \(a_2,b_2\) and adjoin \(i_2\) in the original first-lift graph \(L\). If \(i_2\in K\), the red switch \(i_2a_2\) excludes \(a_2\); without \(b_2\) the projected set is a blue clique in \(G\), and with \(b_2\) replace it by \(c_2\). No collision is possible because all of \(C\) is red to \(w\).

If one small vertex \(b_r\), \(r\in\{0,1\}\), occurs, then \(i_2\) and \(a_2\) are absent. Without \(b_2\), replace \(b_r\) by \(a_r\) and adjoin \(i_2\). With \(b_2\), the switch \(b_0b_2\) forces \(r=1\); replace \(b_1\) by \(a_2\), retain \(b_2\), and adjoin \(i_2\). Each construction gives an injective blue clique of order \(|K|+1\) in \(L\). Thus \(|K|\le s-1\) in every no-\(X\) subcase.

No blue \(K_{s+1}\) containing \(w\) exists. Therefore \(H_B^+\) is a \((k+1,s+1)\)-graph.

\subsection{Return from the selected orientation}

If Lemma~\ref{lem:orientation} selected the original orientation, the preceding construction directly produces a \((k_0+1,s_0+1)\)-graph.

If the lemma selected the dual orientation, the construction produces an \((s_0+1,k_0+1)\)-graph in the complemented color convention. Complement the final graph. The result is again a \((k_0+1,s_0+1)\)-graph. The number of vertices is unchanged.

The Type A first-connector proof in Subsection~\ref{subsec:switching} used only \(k,s\ge5\); it did not use \(k\le s\). Therefore it applies in the dual Type A branch even when its oriented parameters are \(s_0>k_0\). In Type B the required ordering \(k\le s\) is part of Lemma~\ref{lem:orientation}; a dual Type B can occur only when \(k_0=s_0\).

\subsection{Order calculation}

Let the original critical graph have order

\[
N=R(k_0,s_0)-1.
\]

The first lift adds \(2k-2\) vertices, the complementary second lift adds \(2s\), and the two connectors \(z,w\) add two. In either orientation \(\{k,s\}=\{k_0,s_0\}\), so the final order is

\[
\begin{aligned}
N+(2k-2)+2s+2
 &=R(k_0,s_0)-1+2k_0+2s_0\\[0pt]
 &=R(k_0,s_0)+2k_0+2s_0-1.
\end{aligned}
\]

Thus there is a \((k_0+1,s_0+1)\)-graph on \(R(k_0,s_0)+2k_0+2s_0-1\) vertices. Consequently,

\[
R(k_0+1,s_0+1)\ge R(k_0,s_0)+2k_0+2s_0
\qquad(5\le k_0\le s_0).
\qquad\text{(23)}
\]

This proves the theorem. \(\square\)

\section{The bound for \texorpdfstring{\(R(12,12)\)}{R(12,12)}}\label{sec:r1212}

\begin{corollary}\label{cor:r1212}
\[
  R(12,12)\ge1641.
\]
\end{corollary}

\begin{proof}
Revision~18 of the dynamic survey records
\(R(11,11)\ge1597\) \cite{Radziszowski2026}.  Theorem~\ref{thm:main} with
\(k=s=11\) gives
\[
 R(12,12)\ge R(11,11)+44\ge1597+44=1641.
\]
\end{proof}

The executable construction maps the known \(1596\)-vertex
\((11,11)\)-coloring to a \(1640\)-vertex \((12,12)\)-coloring.  The adjacency
matrix is the file \path{bounds/r1212_1640.graph} in the repository.

\section{Lean verification}\label{sec:lean}

The formal development is written in Lean~4 and uses mathlib
\cite{deMouraUllrich2021,Mathlib2020}.  Its public numerical theorems are:

\begin{lstlisting}[language=Lean]
theorem RamseyIncrement.ramsey_increment_plus
    {k s : Nat} (hk : 5 <= k) (hks : k <= s) :
    ramseyNumber k s + 2 * k + 2 * s <=
      ramseyNumber (k + 1) (s + 1)

theorem RamseyIncrement.ramsey_increment_plus_symmetric
    {k s : Nat} (hk : 5 <= k) (hs : 5 <= s) :
    ramseyNumber k s + 2 * k + 2 * s <=
      ramseyNumber (k + 1) (s + 1)
\end{lstlisting}

The formalization includes:

\begin{enumerate}
\item a local definition of finite Ramsey numbers and a constructive proof of
      their existence;
\item the special Xu--Shao--Radziszowski lift, proved directly;
\item extraction of disjoint critical certificates and the four-way
      primal/dual, Type~A/Type~B orientation theorem;
\item the two canonical lifts and the first connector;
\item the Type~A and Type~B edits, second-connector neighborhoods, and all
      old-edge and through-\(w\) clique exclusions; and
\item the cardinality calculation and the ordered and symmetric numerical
      conclusions.
\end{enumerate}

The main statements are in
\texttt{lean/RamseyIncrement/MainTheorem.lean}.  The orientation theorem is in
\texttt{lean/RamseyIncrement/PlusCertificates.lean}, and the numerical
dispatch and branch theorems are in
\texttt{lean/RamseyIncrement/SecondConnector/}.

The repository pins the Lean and mathlib revisions.  Its source audit rejects
\texttt{sorry}, \texttt{admit}, \texttt{native\_decide}, project-defined
axioms, \texttt{unsafe} declarations, and \texttt{partial} declarations.  A
theorem-specific audit permits only \texttt{propext},
\texttt{Classical.choice}, and \texttt{Quot.sound}.  The verification workflow
runs the pinned Lean build, the source and axiom audits, the Python tests, a
\texttt{leanchecker} replay, and a revision-pinned replay with the independent
\texttt{nanoda} kernel.  No AI system is part of the trusted
proof base.

\section{Computational checks}\label{sec:computation}

The repository contains a dependency-free Python implementation in
\texttt{scripts/construct.py}.  For an input \((k,s)\)-coloring on \(n\)
vertices with suitable disjoint certificates, it produces a coloring on
\(n+2k+2s\) vertices and a metadata file recording the orientation,
certificate order, distinguished vertices, and all source-checked switches.
The verifier in \texttt{scripts/verify.py} checks the output for forbidden red
and blue cliques.

The bundled tests include the following constructions.

\begin{center}
\begin{tabular}{@{}lll@{}}
\toprule
Source & Selected branch & Verified output \\
\midrule
Paley \(37\), \((5,5)\) & primal Type~A & \((6,6;57)\) \\
Paley \(101\), \((6,6)\) & primal Type~A & \((7,7;125)\) \\
Bundled Type~B source & primal Type~B & \((6,6;30)\) \\
Synthetic certificate core & dual Type~B & \((6,6;28)\) \\
Synthetic \((5,6)\) core & dual Type~A (oriented parameters \((6,5)\)) & \((6,7;31)\) \\
\bottomrule
\end{tabular}
\end{center}

The formulas were also tested by finite enumeration.  All \(35{,}714\)
admissible \(4\times4\) certificate-incidence matrices for
\((k,s)=(5,5)\) were checked: \(35{,}666\) selected primal Type~A, \(24\)
selected primal Type~B, and \(24\) selected dual Type~B, and every output was
\((6,6)\)-free.  The \(360\) direct-orientation failure matrices of size
\(4\times5\) for \((5,6)\) all selected dual Type~A and produced graphs with
no red \(K_6\) and no blue \(K_7\).  Additional random incidence systems,
irregular cores, and the two Paley examples were also checked.

A separate audit of a non-bundled research workspace checked \(55{,}780\)
strengthened constructions across all four orientation branches; every
construction was clean, with no unresolved cases or counterexamples.  The
full counts and the timeout-retry protocol are recorded in
\texttt{audit/empirical/} in the repository.

These computations are not premises of Theorem~\ref{thm:main}.  They are used
to detect indexing, source-color, implementation, and case-analysis errors.

\section*{Data and code availability}

The Lean sources, executable constructor, verification scripts, examples, and
the \(1640\)-vertex \((12,12)\)-coloring are available in the public
repository \cite{TatarevicRepo2026}.

\end{document}